\documentclass[a4paper,12pt,reqno]{amsart}
\usepackage{amssymb}
\usepackage{xcolor}
\usepackage[pdftex]{graphicx,hyperref}
\usepackage[pagewise]{lineno}

\theoremstyle{plain}
\newtheorem{theorem}{Theorem} 
\newtheorem{corollary}[theorem]{Corollary}
\newtheorem{lemma}[theorem]{Lemma}

\theoremstyle{definition}
\newtheorem{definition}{Definition}

\newtheorem{remark}[definition]{Remark}

\newcommand{\abs}[1]{\left\vert #1\right\vert }

\newcommand\egdef{{\,{\buildrel\rm def\over=}\,}}
\newcommand\bg{\medskip\nobreak}
\newcommand{\bbinom}[2]{\left\langle{{#1}\atop{#2}}\right\rangle}
\numberwithin{equation}{section}

\makeatletter
\@namedef{subjclassname@2020}{%
	\textup{2020} Mathematics Subject Classification}
\makeatother

\title[ New identities obtained from Gegenbauer series expansion]
{{\large   New identities obtained from Gegenbauer series expansion}}
\author[Omran Kouba]{Omran Kouba$^\dag$}
\address{Department of Mathematics \\
	Higher Institute for Applied Sciences and Technology\\
	P.O. Box 31983, Damascus, Syria.}
\email{omran\_kouba@hiast.edu.sy}
\keywords{Gegenbauer polynomials, Legendre polynomials, Chebyshev polynomials of the first and second kind, Gamma function, rising factorial.}
\subjclass[2020]{33C20, 33C45, 42C10.}
\thanks{$^\dag$ Department of Mathematics, Higher Institute for Applied Sciences and Technology.}

\begin{document}
	
	\date{\today}
	\begin{abstract}
		Using the expansion in a Fourier-Gegenbauer series, we prove several identities that extend and generalize known results. In particular, it is proved among other results, that
		\begin{equation*}
			\sum_{n=0}^\infty\frac{1}{4^n}\binom{2n}{n}\frac{z-2n}{\binom{z-1/2}{n}}\binom{z}{n}^3
			=\frac{\tan(\pi z)}{\pi}
		\end{equation*}
	for all complex numbers $z$ such that $\Re(z)>-\frac{1}{2}$ and  $z\notin\frac{1}{2}+\mathbb{Z}$.
	\end{abstract}

\maketitle
\section{Introduction and Notation}
For a complex number $a$ and a nonnegative integer $n$ we define the \textit{rising factorial} $(a)_n$ as follows.
\begin{equation}
	(a)_n\egdef\prod_{0\le j<n}(a+j)=\frac{\Gamma(a+n)}{\Gamma(a)}
\end{equation}
where $\Gamma$ is the well-known Eulerian Gamma function (where in the last equality we assume that $a$ is not $0$ or a negative integer). We will also introduce a notation for $(a)_n/n!$ namely
\begin{equation}
\bbinom{a}{n}\egdef\frac{(a)_n}{n!}=(-1)^n\binom{-a}{n}=\binom{a+n-1}{n}.
\end{equation}

\bg
The Gegenbauer (or ultraspherical) polynomials $(C^{(\lambda)}_n)_{n\ge0}$ can be defined by the generating function
\begin{equation}
	\frac{1}{(1-2xz+z^2)^{\lambda}}=\sum_{n=0}^\infty C^{(\lambda)}_n(x) z^n, \quad -1/2<\lambda\neq 0, \, |z|<1.
\end{equation}
They are orthogonal polynomials on $[-1,1]$ with respect to the weight function $\omega_\lambda(x)=(1-x^2)^{\lambda-1/2}$. 
Furthermore,
\begin{align}\label{eqNorm}
	\Vert C_n^{(\lambda)}\Vert_2^2&=
	\int_{-1}^1	\left(C^{(\lambda)}_n(x)\right)^2\omega_\lambda(x)\,dx,\nonumber\\
	&=\frac{\sqrt{\pi}\,\Gamma(\lambda+1/2)}{\Gamma(\lambda)}\cdot\frac{1}{n+\lambda}\bbinom{2\lambda}{n}.
\end{align}

Let $f$ be a measurable function on $(-1,1)$ such that
the integral $\int_{-1}^1\omega_\lambda(x)\abs{f(x)}\,dx$ is convergent,
(that is $f$ is integrable with respect to the measure $\omega_\lambda(x)dx$),
we may define the $\lambda$-Gegenbauer coefficients $a_n(f)$ by the formula
\begin{equation}
	a_n(f)=\frac{1}{\Vert C_n^{(\lambda)}\Vert_2^2}\int_{-1}^1f(x) C_n^{(\lambda)}(x)\omega_\lambda(x)\,dx,\qquad n=0,1,2,\ldots.
\end{equation}
Then we may consider the formal expansion of $f$ in a  $\lambda$-Gegenbauer series
\begin{equation}
	S(f)=\sum_{n=0}^\infty a_n(f)C_n^{(\lambda)},
\end{equation}
and the partial sums
\begin{equation}
	S_n(f)=\sum_{k=0}^n a_k(f)C_k^{(\lambda)}.
\end{equation}
In order to settle the question of convergence of these series to the function $f$ we will use two results. First, the ``equiconvergence theorem''  \cite[9.1.2]{Szego} that we recall the statement specialized to the case of $\lambda$-Gegenbauer polynomials for the convenience of the reader.

\begin{theorem}[Szeg\"o] \label{Szego}
	Let $f$ be Lebesgue-measurable in $[-1,1]$, and let the integrals
	\begin{equation}
		\int_{-1}^1(1-x^2)^{\lambda-1/2}\abs{f(x)}\,dx,\quad
		\int_{-1}^1(1-x^2)^{(\lambda-1)/2}\abs{f(x)}\,dx
	\end{equation}
	exist. If $S_n(f)$ denotes the $n$th partial sum of the expansion of $f$ in  $\lambda$-Gegenbauer series, and $\tilde{S}_n(\tilde{f})$ is the $n$th partial sum of the Fourier (cosine) series of $\theta\mapsto \tilde{f}(\theta)\egdef\abs{\sin\theta}^{\lambda}f(\cos\theta)	$ then
	for $x\in(-1,1)$ 
	\begin{equation}
		\lim_{n\to\infty}\left(S_n(f)(x)-(1-x^2)^{-\lambda/2}\tilde{S}_n(\tilde{f})(\arccos(x))\right)=0
	\end{equation}
	Moreover, the convergence is uniform in $[-1+\epsilon,1-\epsilon]$, where $\epsilon$ is a fixed positive number, $\epsilon<1$.
\end{theorem}

Then, the above theorem is combined with the next theorem \cite[Chapter II, (8.14)]{Zygmund}, (that we also recall the statement), to conclude the convergence of the $\lambda$-Gegenbauer series expansion.

\begin{theorem}[Zygmund]\label{Zygmund}
	Suppose that $f$ is integrable, $2\pi$-periodic, and of bounded variation in an interval $I$. Then the Fourier series of $f$ converges to ${\frac{1}{2}[f(x^{+})+f(x^{-})]}$ at every point $x$ interior to $I$. If, in addition, $f$ is continuous in $I$, the convergence is uniform in any interval interior to $I$. 
\end{theorem}

\bg
When $\lambda=1/2$, $C_n^{(1/2)}(x)=P_n(x)$ the Legendre polynomial of degree $n$. Legendre polynomials  play an important role in our development  because most of the formulas take an aesthetically beautiful appearance in this case. 

\bg
Chebyshev polynomials  $(T_n)_{n\ge0}$ and  $(U_n)_{n\ge0}$ of the first and second kind, respectively, are also special types of Gegenbauer polynomials, but we prefer to define them by the functional identities:
\begin{equation}
	T_n(\cos\theta)=\cos\theta,\quad U_n(\cos\theta)\sin\theta=\sin((n+1)\theta),
	\quad(\theta\in\mathbb{R}).
\end{equation}

For detailed information on these polynomials we refer the reader to Ismail \cite[Chapter 4]{I},  Milovanovi\'c et al. \cite[Chapter 1.2]{MMR},
\cite[Chapter 18]{Nist} and the references cited therein. 

\bg

In Theorem \ref{T3} and  Theorem \ref{T4}, which are the cornerstones of our investigation, it is proved that, for $\abs{x}<1$, one has
\begin{align}
\frac{ U_m(x)}{(1-x^2)^{\lambda-1}}&=
\frac{(m+1)\sqrt{\pi}\,\Gamma(\lambda)}{2\Gamma(\lambda+1/2)}
\sum_{n=0}^\infty\frac{\lambda+m+2n}{m+n+1} \frac{\bbinom{\lambda-1}{n}\bbinom{\lambda}{m+n}}{\bbinom{2\lambda}{m+2n}}
C_{m+2n}^{(\lambda)}(x)\label{eqU}\\
\noalign{\noindent\text{and}}
\frac{T_m(x)}{(1-x^2)^\lambda}&=\frac{\sqrt{\pi}\,\Gamma(\lambda)}{\Gamma(\lambda+1/2)}
	\sum_{n=0}^\infty(\lambda+m+2n) \frac{ \bbinom{\lambda}{n}\bbinom{\lambda}{m+n}}{\bbinom{2\lambda}{m+2n}} C_{m+2n}^{(\lambda)}(x)\label{eqT}
\end{align}
where $\lambda\in(-1/2,3)$ in \eqref{eqU} and $\lambda\in(-1/2,1)$ in \eqref{eqT}. The particular case corresponding to $m=0$ and $\lambda=1/2$ can be found in the literature. Levrie's paper \cite{Levrie} used this expansion to find formulas for $1/\pi$ and $1/\pi^2$, this was considered again by Carantini and D'Aurizio in \cite{Cantarini}, and more recently by Chen \cite{chen}. Let us emphasize on the statement of convergence. Note that
\begin{equation*}
	\lim_{n\to\infty}	(\lambda+m+2n) \frac{ \bbinom{\lambda}{n}\bbinom{\lambda}{m+n}}{\bbinom{2\lambda}{m+2n}}=\frac{\Gamma(2\lambda)}{(\Gamma(\lambda))^2}.
\end{equation*}
So, it is not straightforward to conclude about the convergence in \eqref{eqT}. In addition, the $L^2$ theory does not apply in the whole domain of $\lambda$. In particular, when $\lambda=1/2$ the function $x\mapsto(1-x^2)^{-1/2}$ is {\it not} square-integrable on $(-1,1)$.

Using \eqref{eqU} and \eqref{eqT}, several series expansions of known functions are obtained using Parseval's identity. In particular, one appealing formula is obtained, (Corollary \ref{T12}),
		\begin{equation}
	\sum_{n=0}^\infty\frac{1}{4^n}\binom{2n}{n}\frac{z-2n}{\binom{z-1/2}{n}}\binom{z}{n}^3
	=\frac{\tan(\pi z)}{\pi}
\end{equation}
which is valid for $\Re(z)>-1/2$ and $z\notin 1/2+\mathbb{Z}$.

\bg
The paper is organized as follows. In section \ref{sec2}, two functions related to Chebyshev polynomials are expanded in a $\lambda$-Gegenbauer series, and the convergence of the corresponding series is studied. In section \ref{sec3}, when the considered functions are square-integrable with respect to weighted measure $\omega_\lambda\,dx$ Parseval's theorem is applied. In section \ref{sec4} many results and applications are demonstrated. We were not exhaustive in our search here because we wanted to keep the length of the paper reasonable.
 
\section{Two expansions in terms of Gegenbauer polynomials}\label{sec2}

Let $\lambda\in(-1/2,3)$ with $\lambda\ne0$, and let $m$ be a positive integer, we define the function $f_{\lambda,m}$ on
$(-1,1)$ by
\begin{equation}\label{eqf}
	f_{\lambda,m}(x)=\frac{ U_m(x)}{(1-x^2)^{\lambda-1}}.
\end{equation}
where $U_m$ is the Chebyshev polynomial of the second kind. 
The function $f_{\lambda,m}$ belongs to $L^1((-1,1),\omega_\lambda\,dx)$. Our objective is to expand $f_{\lambda,m}$ in a $\lambda$-Gegenbauer series.

\begin{theorem}\label{T3}
	Let $m$ be a nonnegative integer and $\lambda\in(-1/2,3)$ with $\lambda\ne0$. Then for all $x\in(-1,1)$ the following holds
	\begin{equation*}\label{T3_1}
		\frac{ U_m(x)}{(1-x^2)^{\lambda-1}}=
		\frac{(m+1)\sqrt{\pi}\,\Gamma(\lambda)}{2\Gamma(\lambda+1/2)}
		\sum_{n=0}^\infty\frac{\lambda+m+2n}{m+n+1} \frac{\bbinom{\lambda-1}{n}\bbinom{\lambda}{m+n}}{\bbinom{2\lambda}{m+2n}}
		C_{m+2n}^{(\lambda)}(x)
	\end{equation*}
	Moreover, the convergence is uniform with respect to $x$ on every compact contained in $(-1,1)$.	
\end{theorem}
\begin{proof}

 The formal expansion of $f_{\lambda,m}$ in a $\lambda$-Gegenbauer series is given by
\begin{equation}
	S(f_{\lambda,m})(x)=\sum_{n=0}^\infty a_n(f_{\lambda,m}) C_n^{(\lambda)}(x)
\end{equation}
with
\begin{align}\label{eq23}
	\Vert C_n^{(\lambda)}\Vert_2^2\, a_n(f_{\lambda,m})&=
	\int_{-1}^1f_{\lambda,m}(x)C_n^{(\lambda)}(x)\omega_\lambda(x)\,dx\nonumber\\
	&=	\int_{-1}^1 U_m(x)C_n^{(\lambda)}(x)\sqrt{1-x^2}\,dx
\end{align}
Now because 
\begin{equation}
	U_m(-x)C_n^{(\lambda)}(-x)=(-1)^{n+m}U_m(x)C_n^{(\lambda)}(x),
\end{equation}
we see that $ a_n(f_{\lambda,m})=0$ if $n+m$ is odd. In addition, because the $(U_m(x))_{m\ge0}$ are orthogonal in $L^2((-1,1),\sqrt{1-x^2}\,dx)$ we conclude that $ a_n(f_{\lambda,m})=0$ if $m>n$. So, we only need to consider the case where $n=m+2q$ for some nonnegative integer $q$.

Now, using the following formula \cite[18.5.10]{Nist} which gives the Fourier series expansion of $\theta\mapsto C_n^{(\lambda)}(\cos\theta)$:
\begin{equation}\label{eq25}
	C^{(\lambda)}_{n}(\cos\theta)
	=\sum_{k=0}^{n} \bbinom{\lambda}{k}\bbinom{\lambda}{n-k} \cos((2k-n)\theta).
\end{equation}
we get from \eqref{eq23}:
\begin{align*}
	\Vert C_{m+2q}^{(\lambda)}\Vert_2^2\, a_{m+2q}(f_{\lambda,m})
	&=\int_{0}^\pi C_{m+2q}^{(\lambda)}(\cos\theta)\sin((m+1)\theta)\sin\theta \,d\theta\nonumber\\
	&=\frac{1}{2}
	\int_{0}^\pi C_{m+2q}^{(\lambda)}(\cos\theta)\cos(m\theta)\,d\theta\nonumber\\
	&\phantom{=}-\frac{1}{2}
	\int_{0}^\pi C_{m+2q}^{(\lambda)}(\cos\theta)\cos((m+2)\theta) \,d\theta\\
	&=\frac{\pi}{2} \bbinom{\lambda}{q}\bbinom{\lambda}{m+q} -
	\frac{\pi}{2}\bbinom{\lambda}{q-1}\bbinom{\lambda}{m+q+1}\nonumber\\
	&=\frac{\pi}{2}\bbinom{\lambda}{q}\bbinom{\lambda}{m+q}\left(1 -\frac{q}{q+\lambda-1}\cdot
	\frac{q+\lambda+m}{q+m+1}\right)\nonumber
\end{align*}
Finally
\begin{equation*}
	\Vert C_{m+2q}^{(\lambda)}\Vert_2^2\, a_{m+2q}(f_{\lambda,m})=
	\frac{\pi}{2}\frac{m+1}{q+m+1}
	\bbinom{\lambda-1}{q}\bbinom{\lambda}{m+q}.
\end{equation*}
It follows that 
\begin{equation}\label{eq26}
	a_{m+2q}(f_{\lambda,m})=
	\frac{(m+1)\sqrt{\pi}\,\Gamma(\lambda)}{2\Gamma(\lambda+1/2)}
	\frac{m+\lambda+2q}{m+1+q}
	\frac{\bbinom{\lambda-1}{q}\bbinom{\lambda}{m+q}}{\bbinom{2\lambda}{m+2q}}
\end{equation}
and
\begin{equation}
	S(f_{\lambda,m})(x)=	\frac{(m+1)\sqrt{\pi}\,\Gamma(\lambda)}{2\Gamma(\lambda+1/2)}
	\sum_{n=0}^\infty
	\frac{m+\lambda+2n}{m+1+n}
	\frac{\bbinom{\lambda-1}{n}\bbinom{\lambda}{m+n}}{\bbinom{2\lambda}{m+2n}}
	C_{m+2n}^{(\lambda)}(x)	
\end{equation}
Using the facts that
\begin{align*}
	\int_{-1}^1 (1-x^2)^{\lambda-1/2}\abs{ f_{\lambda,m}(x)}\, dx&=
	\int_{-1}^1\abs{U_m(x)} \sqrt{1-x^2}\,dx<+\infty,\\
	\int_{-1}^1 (1-x^2)^{(\lambda-1)/2}\abs{f_{\lambda,m}(x)}\, dx&=
	\int_{-1}^1 \frac{\abs{U_m(x)}}{(1-x^2)^{(\lambda-1)/2}}<+\infty,
\end{align*}
(here we use $\lambda<3$), and that, (according to Theorem \ref{Zygmund}), the Fourier series of the function
$\theta\mapsto U_m(\cos\theta)\abs{\sin\theta}^{2-\lambda}$ is uniformly convergent on every compact interval contained in $(0,\pi)$, we conclude by
Theorem \ref{Szego} that the series $S(f_{\lambda,m})(x)$ is convergent with sum $f_{\lambda,m}(x)$ for all $x\in(-1,1)$, the convergence is uniform with respect to $x$ on every compact interval contained in $(-1,1)$. 
\end{proof}

\bg
Next, let $\lambda\in(-1/2,1)$ with $\lambda\ne0$, and let $m$ be a positive integer, we define the function $h_{\lambda,m}$ on
$(-1,1)$ by
\begin{equation}\label{eqh}
	h_{\lambda,m}(x)=\frac{ T_m(x)}{(1-x^2)^{\lambda}}.
\end{equation}
where $T_m$ is the Chebyshev polynomial of the first kind. 
The function $h_{\lambda,m}$ belongs to $L^1((-1,1),\omega_\lambda\,dx)$. In the next theorem we expand $h_{\lambda,m}$ in a $\lambda$-Gegenbauer series.

\begin{theorem}\label{T4}
	Let $m$ be a nonnegative integer and $\lambda\in(-1/2,1)$ with $\lambda\ne0$. Then for all $x\in(-1,1)$ the following holds
	\begin{equation*}\label{eqth3_1}
		\frac{T_m(x)}{(1-x^2)^\lambda}=\frac{\sqrt{\pi}\,\Gamma(\lambda)}{\Gamma(\lambda+1/2)}
		\sum_{n=0}^\infty(\lambda+m+2n) \frac{ \bbinom{\lambda}{n}\bbinom{\lambda}{m+n}}{\bbinom{2\lambda}{m+2n}} C_{m+2n}^{(\lambda)}(x)
	\end{equation*}
	Moreover, the convergence is uniform with respect to $x$ on every compact contained in $(-1,1)$.
\end{theorem}

\begin{proof}
The formal expansion of $h_{\lambda,m}$ in a $\lambda$-Gegenbauer series is given by
\begin{equation}
	S(h_{\lambda,m})(x)=\sum_{n=0}^\infty a_n(h_{\lambda,m}) C_n^{(\lambda)}(x)
\end{equation}
with
\begin{align}\label{eq23-1}
	\Vert C_n^{(\lambda)}\Vert_2^2\, a_n(h_{\lambda,m})&=
	\int_{-1}^1h_{\lambda,m}(x)C_n^{(\lambda)}(x)\omega_\lambda(x)\,dx\nonumber\\
	&=\int_{-1}^1\frac{T_m(x)C_n^{(\lambda)}(x)}{\sqrt{1-x^2}}\,dx,\\
	&=\int_{0}^\pi C_n^{(\lambda)}(\cos\theta)\cos(m\theta)\,d\theta\nonumber.
\end{align}
Now because 
\begin{equation}
	T_m(-x)C_n^{(\lambda)}(-x)=(-1)^{n+m}T_m(x)C_n^{(\lambda)}(x),
\end{equation}
 we see that $ a_n(h_{\lambda,m})=0$ if $n+m$ is odd. So let us assume that $n$ and $m$ have the same parity. Moreover, because the $(T_m(x))_{m\ge0}$ are orthogonal polynomials in $L^2((-1,1),dx/\sqrt{1-x^2})$ we conclude that $a_n(h_{\lambda,m})=0$ if $m>n$. Therefore, we only need to consider the case where $n=m+2q$ for some nonnegative integer $q$. Now, using again \eqref{eq25} we find
\begin{align}
	\Vert C_{m+2q}^{(\lambda)}\Vert_2^2\, a_{m+2q}(h_{\lambda,m})&=
	\int_{0}^\pi C_{m+2q}^{(\lambda)}(\cos \theta) \cos(m\theta)\,d\theta\nonumber\\
	&=\pi\bbinom{\lambda}{q}\bbinom{\lambda}{m+q}.
\end{align}
It follows that 
\begin{equation}
	a_{m+2q}(h_{\lambda,m})=\frac{\sqrt{\pi}\,\Gamma(\lambda)}{\Gamma(\lambda+1/2)}
	(\lambda+m+2q)
	\frac{ \bbinom{\lambda}{q}\bbinom{\lambda}{m+q}}{\bbinom{2\lambda}{m+2q}}
\end{equation}
Consequently
\begin{equation}
	S(h_{\lambda,m})=
		\frac{\sqrt{\pi}\Gamma(\lambda)}{\Gamma(\lambda+1/2)}
		\sum_{q=0}^\infty
	(\lambda+m+2q)\frac{ \bbinom{\lambda}{q}\bbinom{\lambda}{m+q}}{\bbinom{2\lambda}{m+2q}} C_{m+2q}^{(\lambda)}(x).
\end{equation}
Now, using the facts that
\begin{align*}
	\int_{-1}^1 (1-x^2)^{\lambda-1/2}\abs{ h_{\lambda,m}(x)}\, dx&=
	\int_{-1}^1\frac{\abs{T_m(x)}}{\sqrt{1-x^2}}\,dx<+\infty,\\
	\int_{-1}^1 (1-x^2)^{(\lambda-1)/2}\abs{h_{\lambda,m}(x)}\, dx&=
	\int_{-1}^1 \frac{\abs{T_m(x)}}{(1-x^2)^{(\lambda+1)/2}}<+\infty,
\end{align*}
and that, (according to  Theorem \ref{Zygmund}), the Fourier series of the function
$\theta\mapsto \cos(m\theta)\abs{\sin\theta}^{-\lambda}$ is uniformly convergent on every compact interval contained in $(0,\pi)$, we conclude by
Theorem \ref{Szego} that the series $S(h_{\lambda,m})(x)$ is convergent with sum $h_{\lambda,m}(x)$ for all $x\in(-1,1)$ and that the convergence is uniform with respect to $x$ on every compact interval contained in $(-1,1)$. 
\end{proof}

\begin{remark}
	The formula in Theorem \ref{T4} can be written, using the duplication formula for the Gamma function, in the following equivalent form.
		\begin{equation*}\label{rm_4}
		\frac{T_m(x)}{(1-x^2)^\lambda}=\sum_{n=0}^\infty(\lambda+m+2n) \binom{m+2n}{n}
		\frac{\Gamma(\lambda+n)\Gamma(\lambda+m+n)}{2^{1-2\lambda}\Gamma(2\lambda+m+2n)} C_{m+2n}^{(\lambda)}(x).
	\end{equation*}
\end{remark}

\bg
\section{Parseval's theorem in action}\label{sec3}
Note that  $f_{\lambda,m}\in L^2((-1,1),\omega_\lambda dx)$ for $\lambda\in(-1/2,5/2)$ and similarly $h_{\lambda,m}\in L^2((-1,1),\omega_\lambda dx)$ for $\lambda\in(-1/2,1/2)$. Thus we may apply Parseval's theorem. But before proceeding we will need to evaluate two integrals. This is what we do in the next Lemma.

\begin{lemma}\label{lm5}${}$\par
\noindent{\bf(a)} For $\Re(\mu)<1/2$ and all nonnegative integers $m$ we have
\begin{align}\label{eqJ}
	J_m(\mu)\egdef\int_0^\pi\frac{\cos(2m\theta)}{\sin^{2\mu}\theta}\,d\theta&=
	\frac{\sqrt{\pi}\,\Gamma(1/2-\mu)}{\Gamma(1-\mu)}\cdot\frac{(\mu)_m}{(1-\mu)_m}.
\end{align}	
\noindent{\bf(b)} For $\Re(\mu)<1/2$ and all nonnegative integers $m$ we have
\begin{align}\label{eqJ}
	K_m(\mu)\egdef\int_0^\pi\frac{\sin((2m+1)\theta)}{\sin^{2\mu+1}\theta}\,d\theta&=
	\frac{\sqrt{\pi}\,\Gamma(1/2-\mu)}{\Gamma(1-\mu)}\cdot\frac{(1+\mu)_m}{(1-\mu)_m}.
\end{align}	
\end{lemma}
\begin{proof}
{\bf(a)} Note that
	\begin{align*}
		J_m({\mu})-J_{m+1}({\mu})&=\int_0^\pi\frac{\cos(2m\theta)-\cos((2m+2)\theta)}{\sin^{2\mu}\theta}\,d\theta\\
		&=2\int_0^\pi \sin((2m+1)\theta)\,\sin^{1-{2\mu}}\theta\,d\theta\\
		&=\left[\frac{-2\cos((2m+1)\theta)}{(2m+1)\sin^{{2\mu}-1}\theta}\right]_0^\pi\\
	&\phantom{=}+\frac{2-{4\mu}}{2m+1}\int_0^\pi\frac{\cos\theta\cos((2m+1)\theta)}{\sin^{2\mu}\theta}\,d\theta\\
		&=\frac{1-{2\mu}}{2m+1}
		\int_0^\pi\frac{\cos(2m\theta)+\cos((2m+2)\theta)}{\sin^{2\mu}\theta}\,d\theta\\
		&=\frac{1-{2\mu}}{2m+1}(J_m(\mu)+J_{m+1}(\mu)).
	\end{align*}
	It follows that $(m+\mu)J_m(\mu)=(m+1-\mu)J_{m+1}(\mu)$. Thus
	\begin{equation*}
		J_{m+1}(\mu)=\frac{m+\mu}{m+1-\mu}J_m(\mu),
	\end{equation*}	
	and consequently
	\begin{equation}
		J_m(\mu)=\frac{(\mu)_m}{(1-\mu)_m} J_0(\mu),
	\end{equation}	
	Finally the change of variables $\sin^2\theta=t$ shows 
	\begin{align}\label{eq36}
		J_0(\mu)&=2\int_0^{\pi/2}\frac{d\theta}{(\sin^2\theta)^\mu}=
		\int_0^1t^{-\mu-1/2}(1-t)^{-1/2}\,dt\nonumber\\
		&=\frac{\Gamma(1/2)\Gamma(1/2-\mu)}{\Gamma(1-\mu)}=\frac{\sqrt{\pi}\Gamma(1/2-\mu)}{\Gamma(1-\mu)}
		\end{align}
This proves {\bf(a)}.

\bg
\noindent{\bf(b)} This is similar to part {\bf (a)}. Indeed,
\begin{align*}
	K_m(\mu)-K_{m-1}(\mu)&=
	\int_0^\pi\frac{\sin((2m+1)\theta)-\sin((2m-1)\theta)}{\sin^{2\mu+1}\theta}\,d\theta\\
	&=2\int_0^\pi\cos(2m\theta)\sin^{-2\mu}\theta\,d\theta\\
	&=\left[\frac{\sin(2m\theta)}{m}\sin^{-2\mu}\theta\right]_0^\pi+\frac{\mu}{m}
	\int_0^\pi\frac{2\sin(2m\theta)\cos\theta}{\sin^{2\mu+1}}\theta\,d\theta\\
	&=\frac{\mu}{m}
	\int_0^\pi\frac{\sin((2m+1)\theta)+\sin((2m-1)\theta)}{\sin^{2\mu+1}}\theta\,d\theta\\
	&=\frac{\mu}{m}(K_m(\mu)+K_{m-1}(\mu))
\end{align*}
It follows that $(m-\mu)K_m(\mu)=(m+\mu)K_{m-1}(\mu)$. Now because $K_0(\mu)=J_0(\mu)$ we conclude that 
\begin{equation*}
	K_m(\mu)=\frac{(1+\mu)_m}{(1-\mu)_m}J_0(\mu),
\end{equation*}
and {\bf(b)} follows using \eqref{eq36}.
\end{proof}

\begin{theorem} ${}$\par\label{T6}
\noindent{\bf(a)} For all $\lambda\in(-1/2,1/2)$ and all nonnegative integers $m$ the following holds
\begin{equation}\label{eqT6_a}
	\sum_{n=0}^\infty(\lambda+m+2n)\frac{\bbinom{\lambda}{n}^2\bbinom{\lambda}{n+m}^2}{\bbinom{2\lambda}{m+2n}}
	=\frac{\tan(\pi\lambda)}{2\pi}\left(1+\frac{(\lambda)_m}{(1-\lambda)_m} \right).
\end{equation}
\noindent{\bf(b)} For all $\lambda\in(-1/2,5/2)$ and all nonnegative integers $m$ the following holds
\begin{equation}\label{eqT6_b}
	\sum_{n=0}^\infty\frac{\lambda+m+2n}{(m+n+1)^2}
	\frac{\bbinom{\lambda-1}{n}^2\bbinom{\lambda}{m+n}^2}{\bbinom{2\lambda}{m+2n}}\\
	=\frac{(1-2\lambda)\tan(\pi\lambda)}{(m+1)^2\pi(1-\lambda)}
	\left(1-\frac{(\lambda-1)_{m+1}}{(2-\lambda)_{m+1}}\right)
\end{equation} \par
where the right side is defined by continuity when $\lambda\in\{1/2,3/2\}$. \par
\noindent{\bf(c)} For all $\lambda\in(-1/2,3/2)$ and all nonnegative integers $m,q$ the following holds
\begin{align}\label{eqT6_c}
	\sum_{n=m}^\infty\frac{\lambda+q+2n}{q+m+1+n} &\frac{\bbinom{\lambda-1}{n-m}\bbinom{\lambda}{n}\bbinom{\lambda}{q+n}\bbinom{\lambda}{q+m+n}}{\bbinom{2\lambda}{q+2n}}\nonumber\\
	&=\frac{(1-2\lambda)\tan(\pi\lambda)}{2\pi(q+2m+1)(1-\lambda)}
	\left(\frac{(\lambda)_{m+q}}{(2-\lambda)_{m+q}}+\frac{(\lambda)_{m}}{(2-\lambda)_{m}}\right)
\end{align}
where the right side is defined by continuity when $\lambda=1/2$. \par
\noindent{\bf(d)} For all $\lambda\in(-1/2,3/2)$ and all nonnegative integers $m,q$ the following holds
\begin{align}\label{eqT6_d}
	\sum_{n=m+1}^\infty
	\frac{\lambda+q+2n}{q+n+1}
	&\frac{\bbinom{\lambda}{n-m-1}\bbinom{\lambda-1}{n}\bbinom{\lambda}{q+n}\bbinom{\lambda}{q+m+1+n}}{\bbinom{2\lambda}{q+2n}}\nonumber\\
	&=
	\frac{(1-2\lambda)\tan(\pi\lambda)}{2\pi(q+1)(1-\lambda)}
	\left(\frac{(\lambda)_{m+q+1}}{(2-\lambda)_{m+q+1}}-\frac{(\lambda)_{m}}{(2-\lambda)_{m}}\right)
\end{align}
where the right side is defined by continuity when $\lambda=1/2$. \par
\end{theorem}
\begin{proof}
{\bf(a)} We know that for $\lambda\in(-1/2,1/2)$ the function $h_{\lambda,m}$ defined in \eqref{eqh} belongs to $L^2((-1,1),\omega_\lambda dx)$ and has the expansion given in Theorem \ref{T4} in terms of $\lambda$-Gegenbauer polynomials. But
\begin{align}
	\Vert h_{\lambda,m}\Vert_2^2&=\int_{-1}^1h_{\lambda,m}^2(x)\omega_\lambda(x)\,dx,\nonumber\\
	&=\int_{-1}^1\frac{(T_m(x))^2}{(1-x^2)^{\lambda}}\,\frac{dx}{\sqrt{1-x^2}},\nonumber\\
	&=\frac{1}{2}\int_0^\pi\frac{1+\cos(2m\theta)}{\sin^{2\lambda}\theta}\,d\theta=\frac{J_0(\lambda)+J_m(\lambda)}{2},
\end{align}
where $J_m$ was defined and calculated in Lemma \ref{lm5}.	
It follows that,
\begin{equation}\label{eq38}
	\Vert h_{\lambda,m}\Vert_2^2=
	\frac{\sqrt{\pi}}{2}\frac{\Gamma(1/2-\lambda)}{\Gamma(1-\lambda)}\left(1+\frac{(\lambda)_m}{(1-\lambda)_m}\right).
\end{equation}
On the other hand according to Parseval's theorem, Theorem \ref{T4} and  \eqref{eqNorm}, we have
\begin{align}\label{eq39}
	\Vert h_{\lambda,m}\Vert_2^2
	&=\left(\frac{\sqrt{\pi}\,\Gamma(\lambda)}{\Gamma(\lambda+1/2)}\right)^2
	\sum_{n=0}^\infty\left((\lambda+m+2n)
	\frac{\bbinom{\lambda}{n}\bbinom{\lambda}{m+n}}{\bbinom{2\lambda}{m+2n}}\right)^2
 \Vert C_{m+2n}^{(\lambda)}\Vert_2^2,\nonumber\\
&=\frac{\pi\sqrt{\pi}\,\Gamma(\lambda)}{\Gamma(\lambda+1/2)}
\sum_{n=0}^\infty(\lambda+m+2n)\frac{\bbinom{\lambda}{n}^2\bbinom{\lambda}{m+n}^2}{\bbinom{2\lambda}{m+2n}}
\end{align}
 So, from \eqref{eq38} and \eqref{eq39} we conclude that
\begin{equation}
	\sum_{n=0}^\infty(\lambda+m+2n)\frac{\bbinom{\lambda}{n}^2\bbinom{\lambda}{m+n}^2}{\bbinom{2\lambda}{m+2n}}=\frac{\tan(\pi\lambda)}{2\pi}\left(1+\frac{(\lambda)_m}{(1-\lambda)_m}\right),
\end{equation}
where the expression with the Gamma function was simplified using the reflection formula: ${\Gamma(z)\Gamma(1-z)=\pi\csc(\pi z)}$,  see \cite[5.5.3]{Nist}.

\bg
\noindent{\bf(b)} Similarly, for $\lambda\in(-1/2,1/2)$, the function $f_{\lambda,m}$ defined by \eqref{eqf}  belongs to $L^2((-1,1),\omega_\lambda dx)$ and has the expansion given in Theorem \ref{T3} in terms of $\lambda$-Gegenbauer polynomials. Moreover,
\begin{align}
	\Vert f_{\lambda,m}\Vert_2^2
	&=\int_{-1}^1f_{\lambda,m}^2(x)\omega_\lambda(x)\,dx=\int_{-1}^1(1-x^2)^{3/2-\lambda}(U_m(x))^2\,dx,\nonumber\\
	&=\frac{1}{2}\int_0^\pi\frac{1-\cos((2m+2)\theta)}{\sin^{2\lambda-2}\theta}\,d\theta
	=\frac{J_0(\lambda-1)-J_{m+1}(\lambda-1)}{2},
\end{align}
So, according to Lemma \ref{lm5} we have for $\lambda\in (0,1)$:
\begin{equation}
	\Vert f_{\lambda,m}\Vert_2^2=
	\frac{1}{2} 
	\frac{\sqrt{\pi}\,\Gamma(3/2-\lambda)}{\Gamma(2-\lambda)}
	\left(1-\frac{(\lambda-1)_{m+1}}{(2-\lambda)_{m+1}}\right).
\end{equation}
Using Parseval's theorem, Theorem \ref{T3} and \eqref{eqNorm} we obtain
\begin{equation*}\label{eq310}
	\Vert f_{\lambda,m}\Vert_2^2
=\frac{(m+1)^2\pi\sqrt{\pi}\,\Gamma(\lambda)}{4\Gamma(\lambda+1/2)}
	\sum_{n=0}^\infty\frac{\lambda+m+2n}{(m+n+1)^2}
	\frac{\bbinom{\lambda-1}{n}^2\bbinom{\lambda}{m+n}^2}{\bbinom{2\lambda}{m+2n}}.
\end{equation*}
Rearranging and using the reflection formula for the Gamma function the formula in {\bf(b)} is proved.

\bg
\noindent{\bf(c)} Assuming $\lambda\in(-1/2,1/2)$ we know that both $f_{\lambda,p}$ and $h_{\lambda,q}$ belong to $L^2((-1,1),\omega_\lambda dx)$. Moreover, assuming that $p$ and $q$ are of the same parity, we have
\begin{align}\label{eqFH}
	\left\langle f_{\lambda,p},h_{\lambda,q}\right\rangle
	&=\int_{-1}^1f_{\lambda,p}(x)h_{\lambda,q}\omega_\lambda(x)\,dx\nonumber\\
	&=\int_{-1}^1(1-x^2)^{1/2-\lambda}U_p(x)T_q(x)\,dx\nonumber\\
	&=\int_{0}^\pi(\sin\theta)^{1-2\lambda}U_p(\cos\theta)T_q(\cos\theta)\sin\theta\,d\theta\nonumber\\
	&=\int_0^\pi\frac{\sin((p+1)\theta)\cos(q\theta)}{\sin^{2\lambda-1}\theta}\,d\theta\nonumber\\
	&=\frac{1}{2}\int_0^\pi\frac{\sin((p+q+1)\theta)+\sin((p+1-q)\theta)}{\sin^{2\lambda-1}\theta}\,d\theta\nonumber\\
	&=\begin{cases}
		\frac{1}{2}(K_{(p+q)/2}(\lambda-1)+K_{(p-q)/2}(\lambda-1))&\text{if $p\geq q$,}\\
		\\
			\frac{1}{2}(K_{(p+q)/2}(\lambda-1)-K_{(q-p)/2-1}(\lambda-1))&\text{if $p<q$.}\\
	\end{cases}
\end{align}
Thus, for nonnegative integers $q$ and $m$ we have
\begin{align}
	\left\langle f_{\lambda,q+2m},h_{\lambda,q}\right\rangle&=
	\frac{K_{q+m}(\lambda-1)+K_{m}(\lambda-1)}{2}\nonumber\\
	&=
\frac{\sqrt{\pi}\Gamma(3/2-\lambda)}{2\Gamma(2-\lambda)}\left(\frac{(\lambda)_{m+q}}{(2-\lambda)_{m+q}}+\frac{(\lambda)_{m}}{(2-\lambda)_{m}}\right).
\end{align}
But according to Theorem \ref{T3} and Theorem \ref{T4} we have
\begin{align*}
	 f_{\lambda,q+2m}(x)& =
	\frac{(q+2m+1)\sqrt{\pi}\,\Gamma(\lambda)}{2\Gamma(\lambda+1/2)}
	\sum_{n=m}^\infty\frac{\lambda+q+2n}{q+m+n+1} \frac{\bbinom{\lambda-1}{n-m}\bbinom{\lambda}{q+m+n}}{\bbinom{2\lambda}{q+2n}}
	C_{q+2n}^{(\lambda)}(x)\\
	h_{\lambda,q}(x)&
	=\frac{\sqrt{\pi}\,\Gamma(\lambda)}{\Gamma(\lambda+1/2)}
		\sum_{n=0}^\infty(\lambda+q+2n) \frac{ \bbinom{\lambda}{n}\bbinom{\lambda}{q+n}}{\bbinom{2\lambda}{q+2n}} C_{q+2n}^{(\lambda)}(x)
\end{align*}
Thus $\left\langle f_{\lambda,q+2m},h_{\lambda,q}\right\rangle$ is given by
\begin{equation*}
	\frac{(q+2m+1)\pi\sqrt{\pi}\,\Gamma(\lambda)}{2\Gamma(\lambda+1/2)}	
	\sum_{n=m}^\infty\frac{\lambda+q+2n}{q+m+1+n} \frac{\bbinom{\lambda-1}{n-m}\bbinom{\lambda}{n}\bbinom{\lambda}{q+n}\bbinom{\lambda}{q+m+n}}{\bbinom{2\lambda}{q+2n}}.
\end{equation*}
It follows that
\begin{align}
	\sum_{n=m}^\infty\frac{\lambda+q+2n}{q+m+1+n} &\frac{\bbinom{\lambda-1}{n-m}\bbinom{\lambda}{n}\bbinom{\lambda}{q+n}\bbinom{\lambda}{q+m+n}}{\bbinom{2\lambda}{q+2n}}\nonumber\\
	&=\frac{(1-2\lambda)\tan(\pi\lambda)}{2\pi(q+2m+1)(1-\lambda)}
	\left(\frac{(\lambda)_{m+q}}{(2-\lambda)_{m+q}}+\frac{(\lambda)_{m}}{(2-\lambda)_{m}}\right),
\end{align}
which is ${\bf(c)}$ for $\lambda\in(-1/2,1/2)$. For a given $q$ and $m$, let us consider the complex functions
\begin{align*}
A_{n}(z)&=\frac{z+q+2n}{q+m+1+n} \frac{\bbinom{z-1}{n-m}\bbinom{z}{n}\bbinom{z}{q+n}\bbinom{z}{q+m+n}}{\bbinom{2z}{q+2n}},\\
B(z)&=
\frac{(1-2z)\tan(\pi z)}{2\pi(q+2m+1)(1-z)}
\left(\frac{(z)_{m+q}}{(2-z)_{m+q}}+\frac{(z)_{m}}{(2-z)_{m}}\right).
\end{align*}
	
Clearly, $B$ is analytic in the domain $\Omega=\{z\in\mathbb{C}:0<\Re(z)<3/2\}$ with a removable singularity at $z=1/2$. On the other hand, using the fact that $\left(k^{1-z}\bbinom{z}{k}\right)_{k>0}$ converges uniformly on every compact subset of $\mathbb{C}$ to $1/\Gamma(z)$, (this is Gauss' version of the definition of $\Gamma(z)$ as a product  \cite[5.8.1]{Nist}), we see that $\big(n^{4-2z}A_n(z)\big)_{n\ge0}$ converges uniformly on every compact subset
of $\Omega$, (to $2^{z-2z}(z-1)\Gamma(2z)/(\Gamma(z))^4$.) This implies that the series $z\mapsto \sum_{n=m}^\infty A_n(z)$ is convergent to some  analytic function in $\Omega$. Now, because $B(\lambda)=\sum_{n=m}^\infty A_n(\lambda)$ for $\lambda\in(0,1/2)$ we conclude the validity of this equality for $\lambda\in(0,3/2)$ by analytic continuation. This proves {\bf (c)}.

\bg
\noindent{\bf(d)} Similarly to {\bf(c)} it is enough to prove the result for $\lambda\in(-1/2,1/2)$. According to \eqref{eqFH}, for nonnegative integers $p$ and $m$ we have
\begin{align}
	\left\langle f_{\lambda,p},h_{\lambda,p+2m+2}\right\rangle&=
	\frac{K_{p+m+1}(\lambda-1)-K_{m}(\lambda-1)}{2}\nonumber\\
	&=
	\frac{\sqrt{\pi}\Gamma(3/2-\lambda)}{2\Gamma(2-\lambda)}\left(\frac{(\lambda)_{m+p+1}}{(2-\lambda)_{m+p+1}}-\frac{(\lambda)_{m}}{(2-\lambda)_{m}}\right),
\end{align}
but according to Theorem \ref{T3} and Theorem \ref{T4} we have
\begin{align*}
	f_{\lambda,p}(x)& =
	\frac{(p+1)\sqrt{\pi}\,\Gamma(\lambda)}{2\Gamma(\lambda+1/2)}
	\sum_{n=0}^\infty\frac{\lambda+p+2n}{p+n+1} \frac{\bbinom{\lambda-1}{n}\bbinom{\lambda}{p+n}}{\bbinom{2\lambda}{p+2n}}
	C_{p+2n}^{(\lambda)}(x)\\
	h_{\lambda,p+2m+2}(x)&
	=\frac{\sqrt{\pi}\,\Gamma(\lambda)}{\Gamma(\lambda+1/2)}
	\sum_{n=m+1}^\infty(\lambda+p+2n) \frac{ \bbinom{\lambda}{n-m-1}\bbinom{\lambda}{p+m+1+n}}{\bbinom{2\lambda}{p+2n}} C_{p+2n}^{(\lambda)}(x)
\end{align*}
hence $\left\langle f_{\lambda,p},h_{\lambda,p+2m+2}\right\rangle$, is given by
\begin{equation*}
\frac{(p+1)\pi\sqrt{\pi}\,\Gamma(\lambda)}{2\Gamma(\lambda+1/2)}
\sum_{n=m+1}^\infty
\frac{\lambda+p+2n}{p+n+1}
\frac{\bbinom{\lambda}{n-m-1}\bbinom{\lambda-1}{n}\bbinom{\lambda}{p+n}\bbinom{\lambda}{p+m+1+n}}{\bbinom{2\lambda}{p+2n}}.
\end{equation*}
It follows that
\begin{align}
	\sum_{n=m+1}^\infty
	\frac{\lambda+p+2n}{p+n+1}
	&\frac{\bbinom{\lambda}{n-m-1}\bbinom{\lambda-1}{n}\bbinom{\lambda}{p+n}\bbinom{\lambda}{p+m+1+n}}{\bbinom{2\lambda}{p+2n}}\nonumber\\
	&=
	\frac{(1-2\lambda)\tan(\pi\lambda)}{2\pi(p+1)(1-\lambda)}
\left(\frac{(\lambda)_{m+p+1}}{(2-\lambda)_{m+p+1}}-\frac{(\lambda)_{m}}{(2-\lambda)_{m}}\right),
\end{align}
which is the required formula.
\end{proof}

\section{Applications}\label{sec4}

Since our series expansions of theorems \ref{T3} and \ref{T4} are valid in the interval $(-1,1)$, and because we have an explicit evaluation for the considered polynomials at $x=0$, namely, we have 
$C_k^{(\lambda)}(0)=0$ if $k$ is odd and $C_k^{(\lambda)}(0)=(-1)^{k/2}\bbinom{\lambda}{k/2}$ if $k$ is even, see \cite[Table 18.6.1]{Nist}. The next results follow immediately.
\begin{theorem}\label{T7}
	Let $m$ be a nonnegative integer, and $\lambda\in(-1/2,3)$ with $\lambda\ne0$. Then 
\begin{align}\label{T7_1}
\frac{2\Gamma(\lambda+1/2)}{(2m+1)\sqrt{\pi}\,\Gamma(\lambda)}&=
		\sum_{n=m}^\infty(-1)^{n-m}\frac{\lambda+2n}{n+m+1} \frac{\bbinom{\lambda-1}{n-m}	\bbinom{\lambda}{n}\bbinom{\lambda}{m+n}}{\bbinom{2\lambda}{2n}}\\
		&=\sum_{n=m}^\infty\frac{(-1)^{n-m}}{4^n}\frac{\lambda+2n}{n+m+1}\binom{2n}{n} \frac{\bbinom{\lambda-1}{n-m}	\bbinom{\lambda}{m+n}}{\bbinom{\lambda+1/2}{n}}\\
		&=\sum_{n=m}^\infty\frac{(-1)^{n-m}}{4^n\, n!}\frac{\lambda+2n}{n+m+1}\binom{2n}{n-m} \frac{(\lambda-1)_{n-m}	(\lambda)_{m+n}}{(\lambda+1/2)_{n}}
	\end{align}
\end{theorem}

\begin{theorem}\label{T8}
	Let $m$ be a nonnegative integer and $\lambda\in(-1/2,1)$ with $\lambda\ne0$. Then
	\begin{align}\label{eqth8_1}
	\frac{\Gamma(\lambda+1/2)}{\sqrt{\pi}\,\Gamma(\lambda)}&=
		\sum_{n=m}^\infty(-1)^{n-m}(\lambda+2n) \frac{ \bbinom{\lambda}{n-m}\bbinom{\lambda}{n}\bbinom{\lambda}{m+n}}{\bbinom{2\lambda}{2n}} \\
		&=
		\sum_{n=m}^\infty\frac{(-1)^{n-m}}{4^n}(\lambda+2n) \binom{2n}{n}\frac{ \bbinom{\lambda}{n-m}\bbinom{\lambda}{m+n}}{\bbinom{\lambda+1/2}{n}} \\
		&=
		\sum_{n=m}^\infty\frac{(-1)^{n-m}}{4^n\,n!}(\lambda+2n) \binom{2n}{n-m}\frac{ (\lambda)_{n-m}(\lambda)_{n+m}}{(\lambda+1/2)_{n}}
	\end{align}
\end{theorem}

The particular case of $\lambda=1/2$ is interesting. In this case we obtain from Theorem \ref{T7} and Theorem \ref{T8} the following two identities:
\begin{align}
\frac{4}{\pi}&=
\sum_{n=m}^\infty\frac{(-1)^{n-m}(1+4n)(2m+1)}{(n+m+1)(1+2m-2n)} \frac{1}{(64)^n}\binom{2n-2m}{n-m}	\binom{2n}{n}\binom{2m+2n}{m+n},\label{eq47}\\
\frac{2}{\pi}&=
\sum_{n=m}^\infty\frac{(-1)^{n-m}(1+4n)}{(64)^n} \binom{2n-2m}{n-m}	\binom{2n}{n}\binom{2m+2n}{m+n}.\label{eq48}
\end{align}
In the case $m=0$ the equalities \eqref{eq47} and \eqref{eq48} are known. (see \cite{Levrie}, \cite{Cantarini} and \cite{chen}.)

 Because $C_n^{(1/2)}=P_n$ is the Legendre polynomial of degree $n$, setting $\lambda=1/2$ in both Theorem \ref{T3} and Theorem \ref{T4} yields the next results.

\begin{theorem}\label{T9}
	Let $m$ be a nonnegative integer. Then for all $x\in[-1,1]$ the following holds
	\begin{equation*}\label{T9_1}
		\sqrt{1-x^2} U_m(x)=
		\frac{(m+1)\pi}{4^{m+1}}
		\sum_{n=0}^\infty\frac{1+2m+4n}{(1-2n)(n+m+1)16^n} \binom{2n}{n}\binom{2n+2m}{n+m}
		P_{m+2n}(x)
	\end{equation*}
	Moreover, the convergence is uniform.
\end{theorem}
Note that, because $\sup_{[-1,1]}\abs{P_{m+2n}(x)}=1$, the series in theorem \ref{T9} is normally convergent on the interval $[-1,1]$, and equality holds on $[-1,1]$.

\bg
Similarly, 
\begin{theorem}\label{T10}
	Let $m$ be a nonnegative integer. Then for all $x\in(-1,1)$ the following holds
	\begin{equation*}\label{T10_1}
		 \frac{T_m(x)}{\sqrt{1-x^2}}=
		\frac{\pi}{ 2^{2m+1}}
		\sum_{n=0}^\infty\frac{1+2m+4n}{16^n} \binom{2n}{n}\binom{2n+2m}{n+m}
		P_{m+2n}(x)
	\end{equation*}
	Moreover, the convergence is uniform on every compact contained in $(-1,1)$.
\end{theorem}

As a corollary from the above Theorems \ref{T9} and \ref{T10} we have the following result which corresponds to the case $m=0$.
\begin{corollary}\label{co11}
	For all $x\in(-1,1)$ the following holds
\begin{align*}\label{cor11_1}
		\sqrt{1-x^2}&=
\frac{\pi}{4}
\sum_{n=0}^\infty\frac{1+4n}{(1-2n)(n+1)16^n} \binom{2n}{n}^2
P_{2n}(x)\\
 \frac{1}{\sqrt{1-x^2}}&=
\frac{\pi}{ 2}
\sum_{n=0}^\infty\frac{1+4n}{16^n} \binom{2n}{n}^2
P_{2n}(x)
\end{align*}
Moreover, the convergence is uniform on every compact contained in $(-1,1)$.
\end{corollary}

These formulas in Corollary \ref{co11} appeared several times in the literature. With exception of the statement about convergence, they can be found in \cite{Levrie}, \cite{Cantarini} and more recently in \cite{chen}.

The next result follows from Theorem \ref{T6} when some interesting values for $\lambda$ and $m$ are selected.

\begin{corollary}\label{T12} For all complex numbers $z$ with $\Re(z)>-\frac{1}{2}$ and  $z\notin\{\frac{1}{2}+k:k\in\mathbb{Z}\}$ the following holds
\begin{equation}\label{eqT12_a}
	\sum_{n=0}^\infty\frac{1}{4^n}\binom{2n}{n}\frac{z-2n}{\binom{z-1/2}{n}}\binom{z}{n}^3
	=\frac{\tan(\pi z)}{\pi}.
\end{equation}
In particular, for all nonnegative  integers $m$ we have
\begin{equation}\label{eqT12_b}
	\sum_{n=0}^\infty\frac{1}{4^{n+1}}\binom{2n}{n}\frac{4m+1-8n}{\binom{m-1/4}{n}}\binom{m+1/4}{n}^3
	=\frac{1}{\pi}.
\end{equation}
\end{corollary}
\begin{proof}
Indeed, both sides of the given equality are analytic in the domain \begin{equation*}\Omega=\left\{z\in\mathbb{C}:\Re(z)>\frac{1}{2},z-\frac{1}{2}\notin\mathbb{Z}\right\},
\end{equation*}
and according to \eqref{eqT6_a} with $m=0$ both sides are equal when $z=-\lambda\in (-1/2,1/2)$ so they are equal in $\Omega$. 
\end{proof}

\begin{remark}
	In a similar way as in Corollary \ref {T12}, analytic extensions  to a suitable complex domain of all the results in Theorem \ref{T6}, Theorem \ref{T7} and Theorem \ref{T8}, can be obtained. We leave this task to the interested reader.
\end{remark}

Using  \eqref{eqT6_b} with $\lambda=1/2$, we obtain the next corollary.

\begin{corollary}\label{T13}
	For all nonnegative integers $m$ we have
\begin{equation*}\label{eq13_1}
	\sum_{n=0}^\infty\frac{1+2m+4n}{(m+n+1)^2(2n-1)^2}\frac{1}{(256)^{n}}\binom{2n}{n}^2
	\binom{2m+2n}{m+n}^2
	=
	\frac{2^{4m+5}}{\pi^2(2m+1)(2m+3)}.
\end{equation*} 
\end{corollary}

In particular, when $m=0$ we get 
\begin{equation}\label{eq411}
	\sum_{n=0}^\infty\frac{1+4n}{(n+1)^2(2n-1)^2}\frac{1}{(256)^{n}}\binom{2n}{n}^4=
\frac{32}{3\pi^2}.
\end{equation}
This is given in \cite{Levrie},\cite{Cantarini} and \cite{chen}. 

\begin{corollary}\label{T14}
	For all positive integers $m$ we have
	\begin{equation*}\label{eq14_1}
	\sum_{n=0}^\infty\frac{1+2m+4n}{(m+2n)(m+1+2n)}\frac{1}{(256)^n}
\binom{2n}{n}^2\binom{2n+2m}{n+m}^2\\
=\frac{2^{4m+1}}{\pi^2}\cdot\frac{2H_{2m}-H_m}{m^2}
\end{equation*} 
where $H_k=\sum_{j=1}^k1/j$ is the $k$th harmonic number.
\end{corollary}
\begin{proof}
	According to \eqref{eqT6_b} with $\lambda=3/2$ and $m$ replaced by $m-1$ we have
\begin{equation*}
\sum_{n=0}^\infty\frac{1+2m+4n}{(m+2n)(m+1+2n)}\frac{1}{(256)^n}
\binom{2n}{n}^2\binom{2n+2m}{n+m}^2\\
	=\frac{16^m}{\pi m^2}\lim_{\epsilon\to0}\Theta(\epsilon)
\end{equation*} 
with
\begin{equation*}
	\Theta(\epsilon)=\left(\frac{(1/2+\epsilon)_{m}}{(1/2-\epsilon)_{m}}-1\right)\cot(\pi\epsilon)
\end{equation*}
But
\begin{align*}
	\left(\frac12+\epsilon\right)_m&=\left(\frac12\right)_m\prod_{j=1}^m\left(1+\frac{2}{2j-1}\epsilon\right)\\
	&=\left(\frac12\right)_m\left(1+\Big( \sum_{j=1}^m\frac{2}{2j-1}\Big)\,\epsilon+O(\epsilon^2)\right)\\
	&=\left(\frac12\right)_m\left(1+ (2H_{2m}-H_m)\epsilon+O(\epsilon^2)\right).
\end{align*}	
Thus $\Theta(\epsilon)=2(2H_{2m}-H_m)/\pi+O(\epsilon)$ and the required conclusion follows.
\end{proof}	

In particular, with $m=1$ we obtain the following companion formula to \eqref{eq411}:
\begin{equation}\label{eq412}
	\sum_{n=0}^\infty\frac{(3+4n)(1+2n)}{(n+1)^3}\frac{1}{(256)^{n}}\binom{2n}{n}^4=
	\frac{32}{\pi^2}.
\end{equation}

In a similar sprit. Using \eqref{eqT6_c} with $\lambda=1/2$ we obtain the next corollary

\begin{corollary}\label{T15}
	For all nonnegtive integers $m$ and $q$ we have
\begin{align*}\label{eq15_1}
	\sum_{n=m}^\infty\frac{1+2q+4n}{(s+1+n)(2m+1-2n)}\frac{1}{(256)^n} \binom{2n-2m}{n-m}&\binom{2n}{n}\binom{2q+2n}{q+n}\binom{2s+2n}{s+n}\nonumber\\
&=\frac{2^{4q+3}}{\pi^2(2s+1)(2m+1)}
	\end{align*} 
where $s=q+m$.
\end{corollary}

In particular, with $m=0$, we get for arbitrary nonnegative integer $q$:
	\begin{equation*}
	\sum_{n=0}^\infty\frac{(1+2q+4n)(2q+1)}{(q+1+n)(1-2n)}\frac{1}{(256)^n} \binom{2n}{n}^2\binom{2q+2n}{q+n}^2=\frac{2^{4q+3}}{\pi^2},
\end{equation*} 
and with $q=0$, we get for arbitrary nonnegative integer $m$:
\begin{equation*}
	\sum_{n=m}^\infty\frac{(1+4n)(2m+1)^2}{(m+1+n)(2m+1-2n)}\frac{1}{(256)^n} \binom{2n-2m}{n-m}\binom{2n}{n}^2\binom{2m+2n}{m+n}=\frac{8}{\pi^2},
\end{equation*} 
and finally, with $m=q=0$ we get
\begin{equation*}
	\sum_{n=0}^\infty\frac{1+4n}{(1+n)(1-2n)}\frac{1}{(256)^n} \binom{2n}{n}^4=\frac{8}{\pi^2}.
\end{equation*} 

\section{Concluding Remarks}\label{sec5}

Many of our results give exact evaluation of some generalized hypergeometric functions at $1$ or $-1$. For example, using the standard notation for hypergeometric functions (\cite[16.1]{Nist}) equation \eqref{eqT12_a} is equivalent to
\begin{equation*}
	{}_5F_4\left({
		\frac{1}{2},-z,-z,-z,1-\frac{z}{2}\atop
		1,1,-\frac{z}{2},\frac{1}{2}-z
	};1\right)=\frac{\tan(\pi z)}{\pi z}.
\end{equation*}
To the author's knowledge this is new. In the same spirit, Theorem \ref{T8} shows that
\begin{equation*}
	{}_4F_3\left({
		z,\frac{1}{2}+m,1+m+\frac{z}{2},2m+z\atop
		1+2m,\frac{1}{2}+m+z,m+\frac{z}{2}
	};-1\right)=	\frac{4^m\,m!\,\Gamma(\frac{1}{2}+m+z)}{\sqrt{\pi}\Gamma(1+2m+z)}
\end{equation*}	
for every nonnegative integer $m$.

\end{document}